\documentclass[12pt]{amsart}
\usepackage{amscd,amsmath,amsthm,amssymb}
\usepackage{color}
\usepackage{amsfonts,amssymb,amscd,amsmath,enumerate,verbatim,enumitem}
\usepackage{lineno}
 \def\NZQ{\mathbb}               

 \def\ZZ{{\NZQ Z}}

 \def\KK{{\NZQ K}}
 
 \def\frk{\mathfrak}               
 
 \def\pp{{\frk p}}

 \def\mm{{\frk m}}

 \def\G{{\mathcal G}}

\def\Cc{{\mathcal C}}
 \def\Oc{{\mathcal O}}

 \def\opn#1#2{\def#1{\operatorname{#2}}} 
 \opn\chara{char} \opn\length{\ell} \opn\pd{pd} \opn\rk{rk}
 \opn\projdim{proj\,dim} \opn\injdim{inj\,dim} \opn\rank{rank}
 \opn\depth{depth} \opn\grade{grade} \opn\height{height}
 \opn\embdim{emb\,dim} \opn\codim{codim}
 
 \opn\Tr{Tr} \opn\bigrank{big\,rank}
 \opn\superheight{superheight}\opn\lcm{lcm}
 \opn\trdeg{tr\,deg}
 \opn\reg{reg} \opn\lreg{lreg} \opn\ini{in} \opn\lpd{lpd}
 \opn\size{size} \opn\sdepth{sdepth}
 \opn\link{link}\opn\fdepth{fdepth}\opn\lex{lex}
 \opn\tr{tr}
\opn\tp{t} 
 \opn\div{div} \opn\Div{Div} \opn\cl{cl} \opn\Cl{Cl}
 \opn\Spec{Spec} \opn\Supp{Supp} \opn\supp{supp} \opn\Sing{Sing}
 \opn\Ass{Ass} \opn\Min{Min}\opn\Mon{Mon}
 \opn\Ann{Ann} \opn\Rad{Rad} \opn\Soc{Soc}
 \opn\Im{Im} \opn\Ker{Ker} \opn\Coker{Coker} \opn\Am{Am}
 \opn\Hom{Hom} \opn\Tor{Tor} \opn\Ext{Ext} \opn\End{End}
 \opn\Aut{Aut} \opn\id{id}
 
 \opn\nat{nat}
 \opn\pff{pf}
 \opn\Pf{Pf} \opn\GL{GL} \opn\SL{SL} \opn\mod{mod} \opn\ord{ord}
 \opn\Gin{Gin} \opn\Hilb{Hilb}\opn\sort{sort}
 \opn\PF{PF}\opn\Ap{Ap}
 \opn\aff{aff} \opn
 \con{conv} \opn\relint{relint} \opn\st{st}
 \opn\lk{lk} \opn\cn{cn} \opn\core{core} \opn\vol{vol}  \opn\inp{inp} \opn\nilpot{nilpot}
 \opn\link{link} \opn\star{star}\opn\lex{lex}\opn\set{set}
 \opn\width{wd}
 \opn\Fr{F}
 \opn\QF{QF}
 \opn\G{G}
 \opn\type{type}\opn\res{res}
\opn\Stab{Stab}
 \opn\gr{gr}
 
 \def\pot#1#2{#1[\kern-0.28ex[#2]\kern-0.28ex]}

 \opn\dirlim{\underrightarrow{\lim}}
 \opn\inivlim{\underleftarrow{\lim}}
 \def\Implies{\ifmmode\Longrightarrow \else
         \unskip${}\Longrightarrow{}$\ignorespaces\fi}
 \def\implies{\ifmmode\Rightarrow \else
         \unskip${}\Rightarrow{}$\ignorespaces\fi}
 \def\iff{\ifmmode\Longleftrightarrow \else
         \unskip${}\Longleftrightarrow{}$\ignorespaces\fi}

 \let\:=\colon
 \newtheorem{Theorem}{Theorem}[section]
 
 \newtheorem{Corollary}[Theorem]{Corollary}
 \newtheorem{Proposition}[Theorem]{Proposition}

 \let\epsilon\varepsilon
 \let\kappa=\varkappa
 \def\qed{\ifhmode\textqed\fi
       \ifmmode\ifinner\quad\qedsymbol\else\dispqed\fi\fi}
 \def\textqed{\unskip\nobreak\penalty50
        \hskip2em\hbox{}\nobreak\hfil\qedsymbol
        \parfillskip=0pt \finalhyphendemerits=0}
 \def\dispqed{\rlap{\qquad\qedsymbol}}
 
 \opn\dis{dis}
 \def\pnt{{\raise0.5mm\hbox{\large\bf.}}}
 
 \opn\Lex{Lex}

\begin{document}
\title { Stable set rings which are Gorenstein on the punctured spectrum }

\author {Takayuki Hibi and Dumitru I.\ Stamate}

\address{Takayuki Hibi, Department of Pure and Applied Mathematics, Graduate School of Information Science and Technology,
Osaka University, Suita, Osaka 565-0871, Japan}
\email{hibi@math.sci.osaka-u.ac.jp}

\address{Dumitru I. Stamate, Faculty of Mathematics and Computer Science, University of Bucharest, Str. Academiei 14, Bucharest -- 010014, Romania }
\email{dumitru.stamate@fmi.unibuc.ro}

\dedicatory{ }

\begin{abstract}
The non-Gorenstein locus of stable set rings of finite simple perfect graphs is studied.  We describe combinatorially those perfect graphs whose stable set rings are Gorenstein on the punctured spectrum.  In addition, we show that, in general, for Cohen--Macaulay graded algebras, their Cohen--Macaulay type and residue are largely independent.
\end{abstract}

 


\maketitle

\section*{Introduction}
Let $\KK$ be any field, and $G$ a finite simple graph with vertex set $V(G)=[n]=:\{1,\dots, n\}$.
 A subset $W\subset V(G)$  is called {\em stable} (or {\em independent}) if $\{i, j\}$ is not an edge in  $G$ for all $i, j\in W$. 
The {\em stable set ring} of $G$ is the toric ring associated to the stable set polytope of $G$, namely it is the  monomial subalgebra in $K[x_1,\dots, x_n, t]$
$$
{\rm Stab}_\KK(G)=\KK[ (\prod_{i\in W}x_i)\cdot t: W \text{ stable set in }G].
$$

The stable set polytope of a finite simple graph had been studied in classical combinatorics.  Especially, when $G$ is perfect, its facets can be easily described (\cite{C}).  (A finite graph G is called {\em perfect} if, for all induced subgraphs $H$ of $G$ including $G$ itself, the chromatic number of $H$ is equal to the maximal degree of complete subgraphs contained in $H$.)

Considering the degree in $t$, $\Stab_\KK(G)$ is a standard graded $\KK$-algebra. Some of its algebraic properties have been investigated.
Assume $G$ is a perfect graph.  
Then by \cite{OH},  $\Stab_\KK (G)$ is normal, hence Cohen--Macaulay.
In \cite{OHjct} it is proved that $\Stab_\KK(G)$ is a Gorenstein ring if and only if all the maximal cliques in $G$ have the same number of elements.  Recall that a set $C\subset V(G)$ is called a {\em clique} if $\{i,j\}$ is an edge in $G$ for all $i<j$ in $C$.
The nearly Gorenstein property and the class group of $\Stab_\KK(G)$ have been recently studied in \cite{HS}, \cite{HM-class}.

Let $R$ be a Cohen-Macaulay  local ring (or standard graded $\KK$-algebra) with maximal (homogeneous) ideal $\mm$.
The {\em non-Gorenstein locus} of $R$ is the set of prime ideals $\pp$ in $R$ such that $R_\pp$ is not a Gorenstein ring. By \cite[Lemma 2.1]{HHS} this locus is the Zariski closed set of primes containing the canonical trace ideal $\tr(\omega_R)$.
Naturally, the height of the latter  indicates how far $R$  is from a Gorenstein ring. For some classes of toric rings the non-Gorenstein locus has been studied in \cite{HMP}, \cite{Miy-P}.

One says that $R$ is {\em Gorenstein on the punctured spectrum} if $R_\pp$ is Gorenstein for all $\pp \in \Spec(R)\setminus \{\mm\}$, equivalently that $\tr(\omega_R)$ is an $\mm$-primary ideal, or that $\height(\tr(\omega_R))=\dim R$.

The aim of the present paper is to provide in Theorem~\ref{stab-main} a simple combinatorial description of the perfect graphs $G$ such that the ring $R=\Stab_\KK(G)$ is Gorenstein on the punctured spectrum. It turns out that this situation is also equivalent to $\tr(\omega_R)$ being a power of the graded maximal ideal in $R$. Similar characterizations have been obtained by Herzog, Mohammadi and Page in \cite{HMP} for the non-Gorenstein locus of Hibi rings. 
Also, modeling \cite[Corollary 3.12]{HMP} and using simple techniques on combinatorics on finite posets, we construct, for any integers $4\leq a< b$, a finite simple connected perfect graph $G$ such that $\height (\tr(\omega_R)) =a$ and $\dim R= b$ (Proposition \ref{ex:ab}).

On the other hand, we study numerical invariants related to the  Gorenstein property.  Let $R$ be a Cohen--Macaulay positively graded $\KK$-algebra.  It is known that $R$ is Gorenstein if and only if its Cohen--Macaulay type $\tp(R)=\dim_\KK \Ext_R^{\dim R} (\KK, R)$ is equal to $1$ and if and only if its residue $\res(R)=\ell_R(R/\tr(\omega_R))$ is equal to $0$.  We show, however, that
when $R$ is not Gorenstein, these two invariants are largely independent (Proposition \ref{type-res}).

\section{Stable set rings and the non-Gorenstein locus}

 We first recall some terminology and results on stable set rings. Let $G$ be a finite simple graph with vertex set $[n]$. The {\em clique complex} of $G$ is the simplicial complex $\Delta(G)$ whose faces are the  cliques in $G$. Note that $\{i\}$ is a stable set and a clique for all $i\in [n]$.
 One says that $G$ is {\em pure} if $\Delta(G)$ is a pure simplicial complex, i.e. all  maximal cliques have the same number of elements. The graph $G$ is called {\em perfect} if for any induced graph $H$ of $G$, including $G$, the chromatic number equals the maximal cardinality of cliques contained in $H$. When $G$ is perfect, $R=\Stab_\KK(G)$ is the $\KK$-span of the monomials $(\prod_{i=1}^n x_i^{a_i})t^q$ with $\sum_{i\in C} a_i \leq q$ for each  maximal clique $C$  of $G$. Moreover, the canonical module $\omega_{\Stab_\KK (G)}$ is the $\KK$-span of those monomials $(\prod_{i=1}^n x_i^{a_i})t^q$ where $a_i>0$ for all $i$, and $\sum_{i\in C}  a_i \leq q-1$ for each  maximal clique $C$  of $G$.
The $a$-invariant of $R$ is by definition $a(R)=-\min \{i: (\omega_R)_i \neq 0\}$, and when $G$ is a perfect graph $a(\Stab_\KK(G))=-\dim \Delta(G) -2$ by \cite[Corollary 12]{HS}. Since $R$ is an integral domain, by \cite[Lemma 1.1]{HHS} the trace of its canonical ideal can be computed as $\tr (\omega_R)=\omega_R\cdot \omega_R^{-1}$, where $\omega_R^{-1}=\{x \in Q(R):x \omega_R\subseteq R\}$ is the anticanonical ideal of $R$ and the field of fractions of $R$ is $Q(R)=\KK(x_1,\dots, x_n, t)$, see the proof of \cite[Lemma 11]{HS}.

 We recall that the Segre product of  the graded $\KK$-algebras $S=\oplus_{i\geq 0}S_i$ and $T=\oplus_{i\geq 0} T_i$ is the graded ring $S\# T=\oplus_{i\geq 0} (S_i\otimes_\KK T_i)$. This construction is relevant for us since if the  graph $G$ has the connected components $G_1,\dots, G_s$,  then
$$\Stab_\KK(G) \cong \Stab_\KK(G_1)\# \cdots  \# \Stab_\KK (G_s).$$
 
 We refer to \cite{BH}, \cite{BG} for more background on Gorenstein and Cohen-Macaulay toric algebras, to \cite{BB} for the undefined terminology in graph theory and to \cite{TH87} for results on Hibi rings.

\begin{Theorem}\label{stab-main}
Let $G$ be a finite simple pefect graph with connected components $G_1,\dots, G_s$ such that  $d_i=\dim \Delta(G_i)$ with $i=1,\dots, s$ satisfy $d_1\geq \dots \geq d_s$. 
We denote $R=\Stab_\KK(G)$ and $\mm$ the graded maximal ideal in $R$.
The following statements are equivalent:
\begin{enumerate}
\item[(i)] $G_i$ is a pure graph for $i=1,\dots, s$;
\item[(ii)] $\tr(\omega_R)=\mm^{d_1-d_s}$;
\item[(iii)]  $\tr(\omega_R)=\mm^N$ for some integer  $N\geq 0$;
\item[(iv)] $\tr(\omega_R) \supseteq \mm^\ell$ for some $\ell >0$;
\item[(v)] $\tr(\omega_R)$ is an $\mm$-primary ideal.
\end{enumerate}
\end{Theorem}

\begin{proof}
Say $V(G)=\{1,\dots, n\}$.
(i)\implies (ii):  Since $G_i$ is pure, the ring $R_i=\Stab_\KK(G_i)$ is a Gorenstein $\KK$-algebra (by \cite[Theorem 2.1 (b)]{OHjct}) with $a$-invariant equal to $-d_i-2$ (by \cite[Corollary 12]{HS}).
Since $R$ is graded isomorphic to the Segre product of algebras $R_1\#\cdots \# R_s$ and $R$ is a Cohen-Macaulay ring, by \cite[Theorem 2.7]{HMP} (or \cite[Theorem 4.15]{HHS}) we obtain that  $\tr(\omega_R)=\mm^{d_1-d_s}$.

The implications (ii)\implies (iii) \implies (iv) and the equivalence (iv) \iff (v)  are clear. We now prove (iv)\implies (i).

Let $\ell >0$ such that  $\tr(\omega_R) \supseteq \mm^\ell$.
Assume, by contradiction, that the graph $G_i$ is not pure for some $i$. We denote by $\delta$ and $\delta_1$ the maximum size of a clique in $G$, respectively in $G_i$.
Since $G_i$ is not pure, there exists an edge $\{i_0, j_0\}$ in $G_i$ such that $i_0$ belongs to a clique $C$ with $\delta_1$ elements, and $j_0$ belongs to no clique in $G$ with $\delta_1$ elements.

As $\{x_{i_0}\}$ is a stable set in  $G$, one has $x_{i_0}t \in \mm_R$. Then
\begin{equation}\label{firstprod}
 x_{i_0}^\ell t^\ell =w \cdot w',
\end{equation}
 where 
$w=\prod_{j=1}^n x_j^{a_j}\cdot t^q \in \omega_R$ and $w'=\prod_{j=1}^n  x_j^{a'_j} \cdot t^{q'} \in \omega_R^{-1}$.

  The monomial $z=x_1\dots x_n t^{\delta+1}$  satisfies the equations describing the monomial $\KK$-basis of $\omega_R$, hence $z$  is in $\omega_R$. Arguing similarly, it follows that any monomial in $\omega_R$ is divisible by $z$.
Since $zw' \in R$, we get that $a'_1,\dots, a'_n \geq -1$ and $q'\geq -\delta-1$.  Also, by considering the equations of $\omega_R$ one verifies that $x_{j_0}z\in \omega_R$, hence $x_{j_0}zw' \in R$. Note that $\deg_{x_{j_0}} (x_{j_0}zw') \geq 1$, therefore $q'+\delta+1 >0$, equivalently $q'\geq -\delta$.

For $j\neq i_0$, since $a_j \geq 1$, $a'_j\geq -1$ and $a_j+a'_j=0$ we get that $a_j=1$ and $a'_j=-1$. Consequently, \eqref{firstprod} becomes
\begin{equation*}
x_{i_0}^\ell t^\ell=
 (x_1 \cdots x_{i_0}^u \cdots x_n t^{\delta+a})\cdot (x_1^{-1}\cdots x_{i_0}^v\cdots x_n^{-1}t^{-\delta+b})
\end{equation*}
for some integers $u\geq 1$, $v\geq -1$, $a\geq 1$ and $b\geq 0$ such that $u+v=a+b=\ell$.

Set $z_1=z\cdot x_{i_0}^{\delta-\delta_1}$. By our choice of $i_0$ and $j_0$, it follows that $z_1$ and $z_1\cdot x_{j_0}$ are in $\omega_R$. Then the monomial $z_1 x_{j_0}w'=x_{i_0}^{v+1+\delta-\delta_1} \cdot x_{j_0}\cdot t^{b+1} \in R$.
 
As $w \in \omega_R$ and $i_0$ belongs to a clique with $\delta_1$ elements, one gets $u+\delta_1-1 \leq \delta+a-1$, i.e.
\begin{equation}
\label{one}
u \leq \delta-\delta_1+a.
\end{equation}

Let $C$ be a maximal clique in $G$ containing $i_0$, and hence $j_0\notin C$. Since $z_1x_{j_0}w'$ is in $R$ we infer that $v+1+\delta-\delta_1 \leq b+1$, i.e.
\begin{equation}
\label{second}
v\leq  b- (\delta-\delta_1).
\end{equation}
Adding \eqref{one} and \eqref{second} we get that $\ell=u+v \leq a+b=\ell$. Thus \eqref{one} and \eqref{second} are both equalities. This gives $z_1x_{j_0}w'=x_{i_0}^{v+1+\delta-\delta_1}x_{j_0}t^{b+1}=x_{i_0}^{b+1}x_{j_0}t^{b+1}$.

On the other hand,  since $\{i_0, j_0\}\in E(G)$ there exists a maximal clique in $G$ containing $i_0$ and $j_0$. The corresponding inequality for  $z_1x_{j_0}w' \in R$ becomes $b+2\leq b+1$, a contradiction.
This finishes the proof of the theorem.
\end{proof}

According to \cite{HHS}, a standard graded Cohen-Macaulay $\KK$-algebra $R$  with canonical module $\omega_R$ is called {\em nearly Gorenstein} if $\tr(\omega_R) \supseteq \mm$, where $\mm$ denotes the graded maximal ideal of $R$. The case $\tr(\omega_R)=R$ corresponds to $R$ being Gorenstein.

As a corollary of Theorem~\ref{stab-main}, we characterize when is $\Stab_\KK(G)$ a nearly Gorenstein ring, recovering a result in \cite{HS}.

\begin{Corollary} (\cite[Theorem 13]{HS}) In the setting of Theorem \ref{stab-main}, the ring $\Stab_\KK (G)$ is nearly Gorenstein if and only if each $G_i$ is pure and $d_s\leq d_1 \leq d_s+1$.
\end{Corollary}

In \cite[Corollary 3.12]{HMP}, a family of Hibi rings whose non-Gorenstein locus has arbitrarily large dimension is constructed.  Simple techniques on combinatorics of finite posets yield a family of stable set rings whose non-Gorenstein locus has arbitrarily large dimension.

Let $P$ be a finite poset. We denote by $\KK[\Oc(P)]$ and $\KK[\Cc(P)]$ the toric ring of the order polytope, respectively of the chain polytope of $P$, introduced in \cite{Stanley-2poset}. Note that  $\KK[\Oc(P)]$ is the Hibi ring of $P$. It follows from \cite[Theorem 2.1]{HL_2} 
that $\Oc(P)$ and $\Cc(P)$ are unimodularly equivalent (hence their toric rings are isomorphic) when  $P$ possesses no ``$X$" subposet, i.e. no $5$-element subset $$Q_{a,b}^{y,z}(x) = \{a, b, x, y, z\}$$ with $a < x, b < x, x < y, x < z$, where $a$ and $b$ are incomparable in $P$ and where $y$ and $z$ are incomparable in $P$.  
Let $G(P)$ denote the comparability graph 
 of $P$  the vertex set of $G(P)$ is $P$ and the edges correspond to pairs of distinct comparable elements in the poset.  Every comparability graph is perfect.  
The  stable sets in $G(P)$ are the subsets of $P$ where no two elements are comparable (i.e. the antichains), hence the stable polytope of $G(P)$ has the same set of vertices as  the chain polytope $\Cc(P)$ (\cite{Stanley-2poset}).  Thus their toric rings coincide  ${\rm Stab}_\KK(G(P)) = \KK[\Cc(P)]$.  

Consider integers $4\leq a <b$.
Let $P = P_1 \bigoplus {\overline P_2}$ be the finite poset which is discussed in the proof of \cite[Corollary 3.12]{HMP}. 
Namely, $P_1$ is a chain with $b-a-1$ elements, $P_2$ is the union of a singleton and a chain with $a-2$ elements, $\overline{P_2}$ is obtained by adding a minimum element to $P_2$, and $P_1\bigoplus{\overline P_2} $ denotes the ordinal sum of $P_1$ and 
$\overline{P_2}$ where every element in $P_1$ is smaller than any element in $\overline{P_2}$. 
 Then $P$ has no ``$X$" subposet, which implies $\KK[\Oc(P)] = \KK[\Cc(P)]$. 
According to \cite[Corollary 3.12]{HMP}, one has $\height(\tr(\omega_{\KK[\Oc(P)]}))=a$ and $\dim  {\KK[\Oc(P)]}=b$.  
Hence,  the next result follows.

\begin{Proposition}\label{ex:ab}
Given integers $a$ and $b$ with $4 \leq a < b$, there is a finite connected perfect graph $G$ for which $\height(\tr(\omega_{{\rm Stab}_\KK(G)})) = a$ and $\dim {\rm Stab}_\KK(G) = b$.   
\end{Proposition}

\section{Cohen--Macaulay type and residue}
We close the present paper with making a comment about two numerical invariants related to the Gorenstein property. Let $R$ be a Cohen--Macaulay (positively) graded $\KK$-algebra. Its {\em Cohen--Macaulay type} is $$\tp(R)=\dim_\KK \Ext_R^{\dim R} (\KK, R)$$  (see \cite[Definition 1.2.15]{BH}) and its {\em residue} is the length $$\res(R)=\ell_R(R/\tr(\omega_R))$$ (see \cite{HHS}).
It is known that $R$ is a Gorenstein ring if and only if $\tp(R)=1$ if and only if $\res(R)=0$.
However, when $R$ is not Gorenstein, these two invariants are largely independent.

\begin{Proposition}\label{type-res}
Given the integers $a\geq 2$ and $b \geq 1$ there exists a Cohen--Macaulay $\KK$-algebra with Cohen--Macaulay type $a$ and residue $b$. 
\end{Proposition}
  
\begin{proof}
The algebra $$R=\KK[t^{a+1}, t^{b(a+1)+1}, t^{b(a+2)+2}, \dots, t^{b(a+1)+a}] \subset \KK[t]$$ is the semigroup ring associated to the numerical semigroup $$H=\langle a+1, b(a+1)+1, b(a+1)+2,\dots, b(a+1)+a \rangle,$$ and it is Cohen--Macaulay of dimension one.  More precisely,  $H=\{i(a+1): 0\leq i <b\} \cup \{c, c+1,\dots \}$, where $c=b(a+1)$ is the conductor of  the semigroup $H$.
A computation made in \cite[Example 2.5]{HHS-residue} gives  $\res(R)=b$. The Cohen--Macaulay type of $R$ equals the cardinality of the set of pseudo-Frobenius numbers in $H$, which is $$\PF(H)=\{x\in \ZZ\setminus H: x+h\in H \text{ for all } 0\neq h\in H\},$$ see \cite[Section 3]{S-survey}.  
From the description of $H$,  arguing modulo $a+1$ it follows  that $$\PF(H)=\{(b-1)(a+1)+1,\dots, (b-1)(a+1)+a\}.$$ Hence $\tp(R)=|PF(H)|=a$, as desired.
\end{proof}

The semigroup ring discussed above cannot be a standard graded $\KK$-algebra.  It would, of course, be of interest to find a standard graded $\KK$-algebra $R$ with  $\tp(R)=a$ and $\res(R)=b$.

\medskip

\noindent{\bf Acknowledgement.}  The first author was partially supported by JSPS KAKENHI 19H00637.

{}
\end{document}